\documentclass{article}

\PassOptionsToPackage{no-math}{fontspec}
\usepackage{kotex}
\usepackage{kotex-logo}

\usepackage{silence}
\AtBeginDocument{%
  \setmainhangulfont{UnBatang.ttf}[%
    BoldFont=UnBatangBold.ttf,%
    BoldItalicFont=UnBatangBold.ttf,%
    ItalicFont=UnBatang.ttf,%
    SmallCapsFont=UnBatang.ttf%
  ]%
}

\usepackage{mathrsfs}
\usepackage{amsthm, amscd}
\usepackage{amsfonts,amsmath,amssymb,amsthm,enumerate,epsfig,fancyhdr,subfigure}

\usepackage[hidelinks]{hyperref}
\hypersetup{
    colorlinks=true, linkcolor=black, citecolor=black, urlcolor=blue,
    pdfborder={0 0 0},
}
\makeatletter
\providecommand*{\toclevel@Section}{1}
\makeatother

\newtheorem{thm}{Theorem}[section]

\newtheorem{lem}[thm]{Lemma}
\newtheorem{cor}[thm]{Corollary}

\theoremstyle{definition}
\newtheorem{defn}[thm]{Definition}
\newtheorem{remark}[thm]{Remark}
\newtheorem{example}[thm]{Example}

\title{Determining Particular Solutions for Exponential-Polynomial Forcing Terms in Linear Nonhomogeneous Recurrence Relations}

\author{Heesung Shin\thanks{Department of Mathematics, Inha University, 100 Inha-ro, Michuhol-gu, Incheon 22212, Korea. Email: \texttt{shin@inha.ac.kr}}}

\def\abstractcontent{%
This paper develops a systematic method for determining particular solutions of the $k$th-order linear nonhomogeneous recurrence relation
$$a_n + c_1 a_{n-1} + \cdots + c_k a_{n-k} = \sum_{j=1}^J p_j(n){r_j}^n$$
with $n \geq k$, $c_k \neq 0$, $r_j \neq 0$.
Here each $p_j(n)$ is a polynomial.
The main result is the following:
for the characteristic polynomial $c(t)=t^k+c_1t^{k-1}+\cdots+c_k$, if $s_j$ denotes the multiplicity of $r_j$ as a root of $c(t)$ ($s_j=0$ when $r_j$ is not a root), then there exists a particular solution of the form
$q_n=\sum_{j=1}^J b_j(n)n^{s_j}r_j^n$,
where each $b_j(n)$ is a polynomial of the same degree as $p_j(n)$.
This result parallels the method of undetermined coefficients for linear ODEs with constant coefficients and yields a systematic procedure for determining the form of particular solutions.}

\begin{document}

\maketitle

\begin{abstract}
\abstractcontent
\end{abstract}

\section{Introduction}

A recurrence relation defines each term of a sequence in terms of preceding terms.
Such equations play a central role in algorithm analysis, combinatorics, and probability.
In particular, linear recurrences with constant coefficients have a structure analogous to linear ODEs with constant coefficients,
and this correspondence provides important mathematical insight \cite{GKP94, Elaydi05}.

Consider the $k$th-order linear nonhomogeneous recurrence relation
\begin{equation}\label{eq:rec}
  a_n + c_1 a_{n-1} + c_2 a_{n-2} + \cdots + c_k a_{n-k} = f(n),
  \quad n \geq k
\end{equation}
with  $c_k \neq 0$.
Its general solution has the form
$$a_n = x_n + q_n,$$
where $\{x_n\}$ is the general solution of the associated homogeneous recurrence relation and $\{q_n\}$ is one particular solution.
Thus, determining a particular solution is a key step in solving nonhomogeneous recurrences.

The case in which the forcing term $f(n)$ is a finite sum of polynomial-exponential terms,
\begin{equation}\label{eq:forcing}
  f(n) = \sum_{j=1}^J p_j(n)\,{r_j}^n,
\end{equation}
with $r_j \neq 0$, occurs frequently in applications.
Standard undergraduate textbooks \cite{GKP94, Elaydi05, Rosen19, PSLL23} typically treat the method of undetermined coefficients for cases where $p_j(n)$ is constant or $r_j=1$,
but unified treatments for general exponential-polynomial forcing terms are comparatively rare.

The main result of this paper is to establish the following theorem.
\begin{thm}[Main theorem]
Under \eqref{eq:rec} and \eqref{eq:forcing}, that is, 
$$a_n + c_1 a_{n-1} + \cdots + c_k a_{n-k} = \sum_{j=1}^J p_j(n){r_j}^n, \quad n \geq k \quad$$
with $c_k \neq 0$ and $r_j \neq 0$ for all $j$,
let $c(t)=t^k+c_1t^{k-1}+\cdots+c_k$ be the characteristic polynomial, and let $s_j$ be the multiplicity of $r_j$ as a root of $c(t)$ ($s_j=0$ if $r_j$ is not a root).
Here $r_1,\ldots,r_J$ are assumed to be distinct.
Then there exists a particular solution of the form
$$q_n=\sum_{j=1}^J b_j(n)\,n^{s_j}\,{r_j}^n$$
where each $b_j(n)$ is a polynomial of the same degree as $p_j(n)$.
\end{thm}

Unless otherwise stated, all coefficients, constants, and characteristic roots appearing in this paper are assumed to belong to $\mathbb{C}$.

This result directly parallels the ODE method of undetermined coefficients:
for $$a_0y^{(k)}+\cdots+a_ky=x^me^{\alpha x},$$
one seeks
$y_p=b_m(x) x^s e^{\alpha x}$,
where $b_m(x)$ is a polynomial of degree $m$ and $s$ is the multiplicity of $\alpha$ in the characteristic equation.

The paper is organized as follows.
Section 2 reviews basic theory on linear recurrence relations.
Section 3 states and proves the main theorem and concludes with an illustrative example.

\section{Preliminaries}

This section summarizes the basic concepts and results needed later.

\begin{defn}
A recurrence relation of the form
\[
  a_n + c_1 a_{n-1} + c_2 a_{n-2} + \cdots + c_k a_{n-k} = f(n),
  \quad n \geq k,
\]
is called a \emph{$k$th-order linear nonhomogeneous recurrence relation with constant coefficients}.
Here $c_1,\ldots,c_k\in\mathbb{C}$ and $c_k\neq 0$.
If $f(n)\equiv 0$, it is called a \emph{linear homogeneous recurrence relation}.

Its \emph{characteristic polynomial} is
\[
  c(t) = t^k + c_1 t^{k-1} + c_2 t^{k-2} + \cdots + c_k,
\]
and the roots of $c(t)=0$ are called \emph{characteristic roots}.
\end{defn}

\begin{thm}[General solution of linear homogeneous recurrences]\label{thm:hom}
If the characteristic polynomial of a $k$th-order linear homogeneous recurrence factors as
$$c(t)=\prod_{i=1}^{l}(t-\alpha_i)^{m_i}$$
where $\alpha_i$ are distinct complex numbers and $\sum_{i=1}^{l} m_i = k$, then the general solution is
\[
  x_n = \sum_{i=1}^{l}
  \bigl(A_{i,0} + A_{i,1}\,n + \cdots + A_{i,m_i-1}\,n^{m_i-1}\bigr)\,\alpha_i^n.
\]
Here each $A_{i,j}$ is an arbitrary constant.
See \cite[Corollary~2.24]{Elaydi05}.
\end{thm}

Since equation~\eqref{eq:rec} is linear, the following theorem holds.
\begin{thm}[General solution of linear nonhomogeneous recurrences]\label{thm:gen}
If one particular solution $q_n$ of \eqref{eq:rec} is known, then the general solution is $a_n=x_n+q_n$,
where $x_n$ is the general solution of the associated homogeneous recurrence relation.
\end{thm}

Define the shift operator $E$ by $E[a_n]=a_{n+1}$.
Then the recurrence can be written in operator form as 
$$c(E)[a_{n-k}] = E^k[a_{n-k}]+c_1E^{k-1}[a_{n-k}]+\cdots+c_k\cdot 1[a_{n-k}] = f(n).$$
The forward difference operator $\Delta=E-1$ satisfies
$$\Delta[a_n]=a_{n+1}-a_n,$$
and for any polynomial $p(n)$ of degree at most $m$, we have $\Delta^{m+1}[p(n)]=0$.

\section{Main Theorems}

In this section, we establish the main result for exponential-polynomial forcing terms.
We begin with a key lemma for the purely polynomial case.

\begin{lem}[Polynomial forcing]\label{lem:poly}
For the $k$th-order linear nonhomogeneous recurrence relation
\[
  a_n + c_1 a_{n-1} + \cdots + c_k a_{n-k} = p(n),
  \quad n \geq k,
\]
where $m$ is a degree of the polynomial $p(n)$ and 
$s$ is the multiplicity of $t=1$ as a root of the characteristic polynomial $c(t)$,
in particular, $s=0$ when $c(1)\neq 0$,
there exists a particular solution of the form 
$$h_n=b(n)n^s,$$
where $b(n)$ is a polynomial of degree $m$.
\end{lem}

\begin{proof}
Write $c(t)=(t-1)^sd(t)$ with $d(1)\neq 0$.
Apply the annihilator $\Delta^{m+1}=(E-1)^{m+1}$ for polynomials of degree at most $m$ to both sides of
$c(E)[a_{n-k}] = p(n)$.
Since $E$ and $E-1$ commute,
\[
  (E-1)^{m+1}\,c(E)\,[a_{n-k}] = (E-1)^{m+1}\,p(n)=0.
\]
This is a linear homogeneous recurrence with characteristic polynomial
$$(t-1)^{m+1}c(t)=(t-1)^{s+m+1}d(t).$$
By Theorem~\ref{thm:hom}, the part corresponding to $t=1$ in its general solution is
$$(A_0+A_1n+\cdots+A_{s+m}n^{s+m})\cdot 1^n.$$

Among these terms, those corresponding to
\[
  n^0,n^1,\ldots,n^{s-1}
\]
are solutions of
$c(E)[a_{n-k}]=0$ because $(t-1)^s\mid c(t)$.
Hence a particular solution can be chosen using only
\[
  n^s,n^{s+1},\ldots,n^{s+m}.
\]
That is,
\[
  q_n = \bigl(B_0 + B_1 n + \cdots + B_m n^m\bigr)n^s=b(n)n^s.
\]

To prove existence of coefficients $B_0,\ldots,B_m$, we analyze the highest-degree term in
\[
  c(E)[n^{s+m}].
\]
Since $c(E)=(E-1)^s d(E)=\Delta^s d(1+\Delta)$,
\[
  c(E)\bigl[n^{s+m}\bigr]=\Delta^s\!\left[d(1+\Delta)\bigl[n^{s+m}\bigr]\right].
\]
Now
$$d(1+\Delta)[n^{s+m}] = d(1)\,n^{s+m} + (\text{terms of degree }\le s+m-1),$$
and the highest-degree term of $\Delta^s[n^{s+m}]$ is
$$\bigl((s+m)(s+m-1)\cdots(m+1)\bigr)n^m = \dfrac{(s+m)!}{m!}n^m.$$
Therefore
\[
  c(E)\bigl[n^{s+m}\bigr]=d(1)\cdot\frac{(s+m)!}{m!}n^m + (\text{terms of degree }\le m-1).
\]
Since $d(1)\neq 0$, the leading coefficient of $c(E)[n^{s+m}]$ is nonzero.
Hence in $$c(E)[q_{n-k}] = p(n),$$ matching coefficients of
$n^m,n^{m-1},\ldots,n^0$ successively determines
$$B_m,B_{m-1},\ldots,B_0$$ uniquely.
\end{proof}

We now prove the main theorem using this lemma.

\begin{thm}[Particular solution for a single exponential-polynomial forcing term]\label{thm:main}
For the $k$th-order linear nonhomogeneous recurrence relation
\[
  a_n + c_1 a_{n-1} + \cdots + c_k a_{n-k} = p(n)\,r^n,
  \quad n \geq k,
\]
assume that $p(n)$ is a polynomial of degree $m$ and $r\neq 0$.
Let $s$ be the multiplicity of $r$ as a root of the characteristic polynomial $c(t)$ ($s=0$ if $r$ is not a root).
Then there exists a particular solution of the form
\[
  q_n = b(n)\,n^s\,r^n,
\]
where $b(n)$ is a polynomial of degree $m$.
\end{thm}

\begin{proof}
Substitute $a_n=b_nr^n$ (valid since $r\neq 0$). Then the original recurrence becomes
\[
  b_nr^n + c_1b_{n-1}r^{n-1} + \cdots + c_kb_{n-k}r^{n-k} = p(n)r^n.
\]
Divide both sides by $r^{n}$:
\begin{equation}\label{eq:b}
  b_n + \frac{c_1}{r}b_{n-1} + \frac{c_2}{r^2}b_{n-2} + \cdots + \frac{c_k}{r^k}b_{n-k} = p(n).
\end{equation}
The characteristic polynomial of \eqref{eq:b} is
\[
  \widetilde{c}(t)
  = t^k + \frac{c_1}{r}t^{k-1} + \cdots + \frac{c_k}{r^k}
  = \frac{1}{r^k}c(rt).
\]
If $r$ is a root of $c(t)$ with multiplicity $s$, then $t=1=r/r$ is a root of $c(rt)=0$ with multiplicity $s$.
Thus $t=1$ is a root of $\widetilde{c}(t)$ with multiplicity $s$.

Applying Lemma~\ref{lem:poly} to \eqref{eq:b}, we obtain a particular solution of the form
$\widetilde{q_n}=b(n)n^s$, where $b(n)$ has degree $m$.
Hence, since $a_n=b_nr^n$, the desired particular solution is
\[
  q_n=b(n)n^sr^n.
\]
\end{proof}

The next theorem extends the main result to finite sums in the forcing term.

\begin{thm}[Particular solution for finite-sum exponential-polynomial forcing]\label{thm:sum}
Consider the $k$th-order linear nonhomogeneous recurrence relation
\[
  a_n + c_1 a_{n-1} + \cdots + c_k a_{n-k}
  = \sum_{j=1}^J p_j(n)\,r_j^n,
  \quad n \geq k.
\]
Here each $p_j(n)$ is a polynomial and $r_j \neq 0$ for all $j$.
Let $c(t)$ be the characteristic polynomial, and for each $j$, let $s_j$ be the multiplicity of $r_j$ as a root of $c(t)$
($s_j=0$ if $r_j$ is not a root).
Then there exists a particular solution of the form
\[
  q_n = \sum_{j=1}^J b_j(n)\,n^{s_j}\,{r_j}^n,
\]
where each $b_j(n)$ is a polynomial of the same degree as $p_j(n)$.
\end{thm}

\begin{proof}
For each $j$, consider
\[
  a_n + c_1 a_{n-1} + \cdots + c_k a_{n-k} = p_j(n)\,{r_j}^n.
\]
By Theorem~\ref{thm:main}, there exists a particular solution of the form
\[
  q_n^{(j)} = b_j(n)\,n^{s_j}\,r_j^n.
\]
By linearity of recurrence relations,
\[
  \sum_{j=1}^J\bigl(q_n^{(j)} + c_1q_{n-1}^{(j)} + \cdots + c_kq_{n-k}^{(j)}\bigr)
  = \sum_{j=1}^J p_j(n)\,{r_j}^n.
\]
Therefore $q_n:=\sum_{j=1}^J q_n^{(j)}$ is a particular solution of the original recurrence and has the desired form.
\end{proof}

The following corollary summarizes standard cases implied by Theorem~\ref{thm:sum}.

\begin{cor}\label{cor:table}
In \eqref{eq:rec}, let $c(t)$ be the characteristic polynomial and let $s$ denote the multiplicity of $r$ in $c(t)$.
Table~\ref{tab:particular} gives the form of particular solutions for representative forcing terms.
\end{cor}

\begin{table}[t]
\centering
\begin{tabular}{|c|c|c|}
\hline
$f(n)$ & Condition & Form of $q_n$ \\
\hline
$d$ & $c(1)\neq 0$ ($s=0$) & $B$ \\
$d$ & $c(1)=0$ ($s\geq 1$) & $B\,n^s$ \\
$d\,n$ & $c(1)\neq 0$ ($s=0$) & $B_1 n + B_0$ \\
$d\,n$ & $c(1)=0$ ($s\geq 1$) & $(B_1 n + B_0)n^s$ \\
$d\,r^n$ & $c(r)\neq 0$ ($s=0$) & $B\,r^n$ \\
$d\,r^n$ & $c(r)=0$ ($s\geq 1$) & $B\,n^s\,r^n$ \\
$p(n)\,r^n$ & $c(r)\neq 0$ ($s=0$) & $b(n)\,r^n$ \\
$p(n)\,r^n$ & $c(r)=0$ ($s\geq 1$) & $b(n)\,n^s\,r^n$ \\
\hline
\end{tabular}
\caption{Forms of particular solutions according to the forcing term $f(n)$ ($b(n)$ has the same degree as $p(n)$)}
\label{tab:particular}
\end{table}

\begin{example}\label{ex:sum-single}
Find the sequence $\{a_n\}$ satisfying the folllowing recurrence relation with initial conditions:
\[
  a_n-4a_{n-1}+5a_{n-2}-2a_{n-3}=n^3+n^2\,2^n+n\,3^n; \ a_0=0, \ a_1=1, \ a_2=2.
\]
\end{example}

\begin{proof}[Solution.]
The characteristic polynomial is
\[
  c(t)=t^3-4t^2+5t-2=(t-1)^2(t-2).
\]

We write the forcing term as
\[
  f(n)=p_1(n){r_1}^n+p_2(n){r_2}^n+p_3(n){r_3}^n,
\]
\[
  p_1(n)=n^3,\ r_1=1,\quad p_2(n)=n^2,\ r_2=2,\quad p_3(n)=n,\ r_3=3.
\]

In $c(t)$, the multiplicities are $s_1=2$ for $r_1=1$, $s_2=1$ for $r_2=2$, and $s_3=0$ for $r_3=3$.
Thus, by Theorem~\ref{thm:sum}, we set
\[
  q_n=q_n^{(1)}+q_n^{(2)}+q_n^{(3)},
\]
with trial forms
\[
  q_n^{(1)}=(A_0+A_1n+A_2n^2+A_3n^3)n^2,
  \qquad
  q_n^{(2)}=(B_0+B_1n+B_2n^2)n\,2^n,
\]
\[
  q_n^{(3)}=(D_0+D_1n)3^n.
\]

Matching coefficients for the first term gives
\[
  q_n^{(1)}=-\frac{99}{4}n^2-\frac{65}{12}n^3-\frac34n^4-\frac1{20}n^5,
\]
and for the second term gives
\[
  q_n^{(2)}=\left(\frac{74}{3}n-6n^2+\frac43n^3\right){2}^n.
\]
For the third term,
\[
  q_n^{(3)}=\left(-\frac{81}{4}+\frac{27}{4}n\right)3^n.
\]
Hence, one particular solution is
\[
  q_n=-\frac{99}{4}n^2-\frac{65}{12}n^3-\frac34n^4-\frac1{20}n^5
  +\left(\frac{74}{3}n-6n^2+\frac43n^3\right){2}^n
  +\left(-\frac{81}{4}+\frac{27}{4}n\right)3^n.
\]

The general solution of the associated homogeneous recurrence relation is
\[
  x_n=C_1+C_2n+C_3 2^n,
\]
therefore, the general solution of the original recurrence is
\[
  a_n=x_n+q_n=C_1+C_2n+C_3 2^n+q_n.
\]

Substituting the initial conditions yields
\[
  C_1+C_3=\frac{81}{4},
\]
\[
  C_1+C_2+2C_3=\frac{487}{15},
\]
\[
  C_1+2C_2+4C_3=\frac{4481}{60}.
\]
Solving this system yields
\[
  C_1=-\frac{39}{4},\qquad C_2=-\frac{1067}{60},\qquad C_3=30.
\]
Therefore the solution to the linear nonhomogeneous recurrence relation satisfying the given initial conditions is
\[
  \begin{aligned}
    a_n={}&-\frac{39}{4}-\frac{1067}{60}n-\frac{99}{4}n^2-\frac{65}{12}n^3
    -\frac34n^4-\frac1{20}n^5 \\
    &+\left(30+\frac{74}{3}n-6n^2+\frac43n^3\right)2^n
    +\left(-\frac{81}{4}+\frac{27}{4}n\right)3^n.
  \end{aligned}
\]
\end{proof}

\begin{remark}
This example illustrates the superposition principle for forcing terms of the form
$f(n)=n^3+n^2\,{2}^n+n\,3^n$:
a particular solution is constructed as the sum of particular solutions for each term.
It also shows that different multiplicities $2$, $1$, $0$ 
of characteristic roots $1$, $2$, $3$
induce the factors $n^2$, $n$, and $1$, respectively.
\end{remark}


\section*{Acknowledgments}

We acknowledge the assistance of AI language models in refining and correcting the English expressions throughout this paper.




\end{document}